\documentclass[12pt]{article}
\usepackage[english]{babel}
\usepackage{times,amsfonts,amsmath,amsxtra,amsthm,amssymb,latexsym}
\usepackage[cp866]{inputenc}
\usepackage[dvips]{epsfig}
\usepackage{graphicx}

\advance\hoffset -0.5in
\advance\voffset -0.6in

\makeatletter\@addtoreset {equation}{section}\makeatother
\theoremstyle{plain}

\newtheorem{theorem}{Theorem}[section]

\newtheorem{lemma}[theorem]{Lemma}
\newtheorem{proposition}[theorem]{Proposition}

\newcommand{\Dom}{{\rm Dom}}

\renewcommand{\Im}{{\bf Im\,}}

\renewcommand{\Re}{{\bf Re\,}}

\def\CC{\mathbb C}
\def\codim{\rm codim\,}

\def\eps{\varepsilon}
\def \RR{{\mathbb R}}

\def\ZZ{\mathbb Z}

\begin{document}

\title{\bf On computing the instability
index of a non-selfadjoint differential operator
associated with coating and rimming flows}

\author{Almut Burchard, Marina Chugunova \\ {\small Department of
Mathematics, University of Toronto, Canada}}

\date{\today}
\maketitle

\begin{abstract}  We study the problem of finding the
instability index of certain non-selfadjoint fourth order
differential operators that appear as linearizations
of coating and rimming flows, where a thin layer
of fluid coats a horizontal rotating cylinder.
The main result reduces
the computation of the instability index to a finite-dimensional
space of trigonometric polynomials. The proof uses Lyapunov's
method to associate the differential operator with
a quadratic form, whose maximal positive subspace has
dimension equal to the instability
index.  The quadratic form is given by
a solution of Lyapunov's equation, which here takes the
form of a fourth order linear PDE in two variables.
Elliptic estimates for the solution of this PDE
play a key role.  We include some numerical examples.
\end{abstract}


\section{Introduction}

The stability of steady states is a basic question about
the dynamics of any partial differential equation
that models the evolution of a physical system.
Frequently, the first step is to linearize the system about
a given equilibrium. Linearized stability is
determined by the spectrum of the resulting differential operator $A$.
If $A$ has discrete spectrum, an important
quantity is the {\em instability index}, $\kappa(A)$, which
counts the number of eigenvalues in the right half plane (with multiplicity).

In order to numerically evaluate the instability
index of a given differential operator, its computation
should be reduced  to a problem of linear algebra.
Particularly for problems with
periodic boundary conditions, it seems natural
to restrict $A$ to a finite-dimensional space
of trigonometric polynomials.
Under what conditions can $\kappa(A)$ be computed from the
resulting finite matrix?  One difficulty is that the
entries of the infinite matrix
corresponding to the differential operator $A$ grow with
the row and column index, so that any truncation is not
a small perturbation.

If $A$ is a self-adjoint semi-bounded differential operator of even
order, then the computation of its instability index is
well-understood through the classical work of Morse \cite{Morse} who
solved this problem completely in the space of vector functions in
one independent variable. The instability index of $A$ agrees with
the dimension of the positive cone of the corresponding quadratic
form. It is invariant under congruence transformations that replace
$A$ with $T^*AT$.  The instability index can be estimated by
variational methods, or computed directly from the zeroes of the
corresponding Evans function.

Understanding the spectrum of a non-selfadjoint operator
is a much harder problem.
It is not at all obvious how to restrict the computation
of its instability index to a finite-dimensional subspace,
or how to even estimate its dimension.  Furthermore, the
numerical calculation of eigenvalues
can be extremely ill-conditioned even in finite dimensions.
One impressive example is the matrix
$$
A = \left(%
\begin{array}{ccc}
  10^4 + 1 & 10^6 & 10^4 \\
  10^6 & 2 & 10^6 \\
  -(10^4) &  -(10^6) & -(10^4-1) \\
\end{array}%
\right)
$$
The {\tt Matlab} function ${\tt eig}(A)$ gives for the eigenvalues
the numerical results $\lambda_1 = -0.8$,
$\lambda_{2/3}= 2.4 \pm  1.7\,i$,
which suggests an instability index of $\kappa(A)=2$.
However, the accuracy of the computation is poor.
Denoting by $V$ the matrix that
contains the (numerically computed)
eigenvectors in its columns, and by $E$ the diagonal matrix
that contains the (numerically computed) eigenvalues, then
$$
norm\, (A - VEV^{-1}) = 7.6\,.
$$
On the other hand, $A$ is similar to an upper triangular  matrix
$$
A = T \left(
\begin{array}{ccc}
  1 & 10^6 & 10^4 \\
  0 & 2 & 10^6 \\
  0 &  0 & 1
\end{array}
\right)T^{-1}\,,
\quad\mbox{where}\
T = \left(%
\begin{array}{ccc}
  1 & 0 & 0 \\
  0 & 1 & 0 \\
  1 & 0 & 1
\end{array}%
\right)\,,
$$
and we see that actually $\lambda_1=\lambda_2=1$,
$\lambda_3=2$ and $\kappa(A)=3$.
In contrast, the eigenvalues of the symmetric matrix
$$
B = \left(%
\begin{array}{ccc}
  10^4 + 1 & 10^6 & 10^4 \\
  10^6 & 2 & -(10^6) \\
  10^4 &  -(10^6) & -(10^4-1) \\
\end{array}%
\right)
$$
can be determined with the much better computational accuracy
$$
norm \, (B - VEV^{-1}) = 3.6 * 10^{-10}\,.
$$
Note that $B$ differs from $A$ only in the signs of two off-diagonal
entries.
The chance of encountering a matrix with moderately-sized
entries and a badly conditioned eigenvalue problem increases
rapidly with the dimension of the matrix (see \cite{Godunov, Trefethen}).
Such examples demonstrate that the stability problem
for a non-selfadjoint operator cannot be easily solved by direct
computations of the spectrum.

In this paper, we examine the computation of the
instability index for differential operators of the form
\begin{equation} \label{eq:def-A}
A [h] = -h{''''} - \bigl(a(x) h\bigr)'' + \bigl(b(x) h\bigr)' - c(x)h\,,
\end{equation}
acting on $2\pi$-periodic functions. Such operators appear as
linearizations of models for thin liquid films moving on the surface
of a horizontal rotating cylinder. The resulting flows are called
{\bf coating}, if the fluid is on the outside of the cylinder, and
{\bf rimming}, if the fluid is on the inside of a hollow cylinder.
They appear in many applications, including coating of fluorescent
light bulbs when a coating solvent is placed inside a spinning glass
tube, different type of moulding processes and paper productions.

One would expect the flow to become unstable, if the fluid film is
thick enough so that drops of fluid can form on the bottom of the
cylinder (in case of a coating flow) or on its ceiling  (in case of
a rimming flow). In both cases, surface tension and higher rotation
speeds should help to stabilize the fluid, but may also allow for
more complicated steady states.

The operators in Eq.(\ref{eq:def-A}) appear
as linearizations of the flows about steady states,
when the dependence on the longitudinal variable
in the cylinder is  neglected.
Benilov, O'Brien and Sazonov \cite{Benilov1} studied the
convection-diffusion equation
$$
 A[h] =  \frac{d}{dx} \Bigl(h + \eps \sin x \frac{dh}{dx}\Bigr)
$$
with periodic boundary conditions on $[0,2\pi]$,
which corresponds to a singular limit of a rimming flow where
surface tension is neglected.
This operator has remarkable properties:
For $|\eps|<2$, all its eigenmodes are neutrally stable,
but the Cauchy problem
$$
\frac{d}{dt} h = A [h],\  \ \ h(0)=h_0
$$
is ill-posed in any Sobolev and H\"older
space of $2\pi$-periodic
functions. The underlying cause is the sign change of the
diffusion coefficient as $x\to x+\pi$. This phenomenon
of explosive instability of a system with purely imaginary spectrum
was studied analytically by Chugunova, Karabash and Pyatkov
\cite{ChKP08Preprint}, who explained it in terms of the absence of
the Riesz basis property of the set of eigenfunctions. The spectral
and asymptotic properties of $A$
are of interest in operator theory and
were analyzed in \cite{D07,Weir1,PT-sym,ChugStr}.

One should expect the explosive instability to disappear in complete
models that includes the smoothing effect of surface tension. Such
models have been proposed, for example, by \cite{pukhnachev1,
pukhnachev2}. In \cite{Benilov2, Benilov3}, the  authors linearized
this model about some approximation of a positive steady state
solution to obtain
\begin{equation}
\label{eq:operator}
 A[h](x) = - \frac{d}{d x} \left\{(1 - \alpha_1 \, \cos x)h
+  \alpha_2 \, \sin x \frac{d h}{d x}  + \alpha_3 \,
      \left( \frac{ d h}{d x} + \frac{d^3 h}{d x^3} \right) \right\}
\end{equation}
with periodic boundary conditions.
Here, the parameter $\alpha_1$ is related to the
gravitational drainage, $\alpha_2$ is related to the hydrostatic
pressure (in lubrication approximation model this coefficient is
very small), and the parameter $\alpha_3$ describes surface tension
effect.  They showed numerically that a sufficiently strong surface tension
can stabilize the film provided that the other coefficients
are not too small.  For smaller values of $\alpha_1$ and $\alpha_2$,
capillary effects destabilize the film. The number of unstable
eigenvalues of $A$ grows if $\alpha_3$ is decreased.

We will consider  operators given by Eq.~(\ref{eq:def-A})
acting on $L^2[0,2\pi]$ with periodic boundary
conditions. We assume that the
coefficients $a(x)$, $b(x)$ and $c(x)$ are bounded smooth
periodic functions.  We will show
that the instability index of $A$ is determined by
its projection to a sufficiently large
finite-dimensional subspace of $L^2$.
The dimension of the space depends on  a suitable
norm of the distributional solution $U$
of the partial differential equation
\begin{equation}\label{eq:PDE-U}
{\cal A^*} U(x,y)= \delta_{y-x}\,
\end{equation}
with periodic boundary conditions on $[0,2\pi]\times [0,2\pi]$.
Here, the differential operator ${\cal A}$
is defined by applying the single-variable differential operator
$A$ to the functions $F(\cdot, y)$ and $F(x,\cdot)$ and
adding the results; symbolically
$$
{\cal A} = {\cal A}_x + {\cal A}_y\,.
$$
We note that Eq.~(\ref{eq:PDE-U}) has a unique solution
if the spectra of $A$ and $-A^*$ are disjoint~\cite{Belonosov1}.
Let $U_0(x,y)$ be the solution of Eq.~(\ref{eq:PDE-U})
with $a(x)=b(x)=0$ and $c(x)=1$.  We will see
below that $U_0$ is piecewise smooth, with a jump in the third
derivative across the line $x=y$, and that
$U(x,y)-U_0(x,y)\in {\cal H}^4$.

To describe our results,
denote by $P_N$ the standard projection onto the space of
trigonometric polynomials of order $N$,
\begin{equation} \label{eq:def-P}
P_N[\phi](x) = \sum_{|p|<N}\hat \phi(x) e^{ipx}\,.
\end{equation}
In Proposition~\ref{prop:U22} we show that
\begin{equation}\label{eq:main-1}
\kappa(A)= \kappa\bigl(P_N U^{-1}P_N\bigr)\,,
\end{equation}
provided that
$$
N^2 > 2 M \bigl(1+ ||U(x,y)-U_0(x,y)||_{{\cal H}^4}\bigr)\,.
$$
The constant is given by
\begin{equation}\label{eq:def-M}
M= 0.52\,\bigl( ||a||_{H^1} + ||b||_{H^1} + ||c-1||_{H^1}\bigr)\,.
\end{equation}
The significance of Eq.~(\ref{eq:main-1}) is that
it allows to compute the instability index of $A$
from the finite matrix that describes the restriction of $U$
to the finite-dimensional subspace
\begin{equation}\label{eq:U-complement}
\bigl(\mbox{Range}\,(I-P)\bigr)^{\perp_U}
= \mbox{Nullspace}\,(P_NU)\,.
\end{equation}
The weakness of this result is that both the
condition on $N$ and the computation of the subspace
involve the unknown function $U$, which is defined
as the solution of a partial differential equation.
The existence of such a solution,  and its norm,
depend sensitively on the spectrum of $A$, which
is exactly the unknown quantity we are concerned with.

It is tempting to consider instead the
matrix obtained by truncating
the Fourier representation of $A$ at a suitable
high order $N$.  Our main result,
Proposition~\ref{prop:main}, guarantees that
\begin{equation}\label{eq:main-2}
\kappa(A)= \kappa\bigl(P_N A P_N\bigr)\,.
\end{equation}
provided that
$$
N^2 > M(1+\sqrt{M+1}) \bigl(1+ ||U(x,y)-U_0(x,y)||_{{\cal H}^4}\bigr)\,.
$$
Note that only the norm of the unknown function
enters into the condition on $N$, and that
the identity in Eq.~(\ref{eq:main-2}) does not
involve $U$ at all.

The selection of $N$ and the problem of estimating this norm will be
discussed at the end.

Let us add a few words about the proofs.
Our analysis relies on the indefinite quadratic form
defined by the self-adjoint operator $U$.
Classical results, which will be discussed in the
next section, state that
$$
\kappa(A)=\kappa(U)\,,
$$
and that the positive and negative cones of $U$ contain the
invariant subspaces associated with
the spectrum of $A$ in the right and
left half planes, respectively.
The key to Eq.~(\ref{eq:main-1}) is that the quadratic
form is negative on high Fourier modes, because the
fourth order term in $A$ dominates the lower order
derivatives.  As part of the argument, we
derive an addition formula for the instability
index of a self-adjoint operator in terms
of its restriction to suitable subspace.
The proof of Eq.~(\ref{eq:main-2}) combines
Eq.~(\ref{eq:main-1}) with estimates
for the off-diagonal terms in the Fourier
representation for $U$.

One of the possible extensions of our results could be an
application of a similar method to obtain the estimations on the
size of the finite dimensional truncation in the case of a more
general forth order differential operator with the third order
derivative term which is absent in \ref{eq:operator}.

\section{Lyapunov's equation}
\label{sec:Lyap}

The partial differential equation~(\ref{eq:PDE-U})
is an instance of {\bf Lyapunov's equation}
\begin{equation} \label{eq:Lyapunov}
A^* U + U A = V\,,
\end{equation}
which was first considered
by Lyapunov in the case where
$A$ and $U$ are $n\times n$ matrices, and $V$ is
symmetric and positive definite.  (In Eq.~(\ref{eq:PDE-U}),
$V=I$.)
Assuming that a symmetric matrix $U$
solves Eq.~(\ref{eq:Lyapunov}), Lyapunov proved that
all eigenvalues of $A$ have
negative real part, if and only if $U$ is negative
definite. The follwoing generalization is due
to Taussky~\cite{Taussky}.

\begin{theorem}[Taussky]
\label{Taussky-theorem}
Let $A$ be an $n\times n$ complex matric with
characteristic roots $\alpha_i$, with
$\alpha_i+\bar\alpha_k\ne 0$ for all $i,k=1,\dots n$.
Then the unique solution $U$ of Lyapunov's equation with $V=I$
is nonsingular and satisfies $\kappa(U)=\kappa(A)$.
\end{theorem}

The problem of
obtaining information about the sign of eigenvalues of
$A$ in situations where both $V$ and $U$ may be
indefinite and have non-trivial kernels remains an
area of active research.

Lyapunov's equation has many applications in stability theory
and optimal control.  In typical applications, $\kappa(A)=0$, so
that the system is asymptotically stable, and $U$ is used to
estimate the rate of convergence.  Eq.~(\ref{eq:Lyapunov})
is a special case of {\bf Sylvester's equation}
$$
AX - XB = C\,,
$$
which has been studied extensively in Linear Algebra,
Operator Theory, and Numerical Analysis.
It is known to be uniquely solvable, if and only if
the matrices $A$ and $B$ have no eigenvalues
in common.  In particular, Eq.~(\ref{eq:Lyapunov}) has a unique
solution if the spectra of $A$ and $-A^*$ are disjoint.
Since $V$ is self-adjoint, a unique solution $U$ is
automatically self-adjoint as well.  These results were extended to
bounded operators on infinite-dimensional
Hilbert spaces by Daleckii and Krein~\cite{Daleckii}
and to unbounded operators by Belonosov \cite{Belonosov1,
Belonosov2}.

Before stating Belonosov's result, we recall that a closed
densely defined operator $A$ on a Banach space is {\bf sectorial}, if
the spectrum of $A$ is contained in an open sector
$$
S= \bigl\{z\in \CC\ \,\big\vert \, |\arg(\lambda_0-z)| < \theta\bigr\}
$$
with vertex at $\lambda_0\in \RR$ and opening angle $\theta<\pi/2$,
and the resolvent $R_\lambda(A) = (A - \lambda I)^{-1}$
is uniformly bounded for $\lambda$ outside $S$.
Sectorial operators are precisely the generators
of analytic semigroups.  The sector $S$
is invariant under similarity transformations,
and does not change if the norm on the
space is replaced by a equivalent norm.

\begin{theorem}[Belonosov]
\label{belonosov-theorem} Let $A$ be a sectorial operator
on a separable Hilbert space $H$.  Assume that
$$
\sigma(A)\cap \sigma(-A^*)=\emptyset\,.
$$
Then for any bounded operator $V$ on $H$, the Lyapunov
equation~(\ref{eq:Lyapunov}) has a unique solution $U$ in the class
of bounded operators on $H$. Then $U$ is invertible in the general
sense, i.e. its inverse is densely defined but can be unbounded
operator
$$
\kappa(A)=\kappa(U)=\kappa(U^{-1})\,.
$$
\end{theorem}

Belonosov actually proved more general existence and uniqueness
results for the Sylvester's equation in Banach spaces.

To explain the geometric meaning of Lyapunov's equation, we
introduce on $H$ the indefinite inner product
\begin{equation}
\label{eq:innerproduct}
[\phi,\psi] = \langle U\phi, \psi\rangle\,.
\end{equation}
If $U$ has trivial nullspace and $\kappa(U)<\infty$,
then  $H$ equipped with $[\cdot,\cdot]$ is called a
{\bf Pontryagin space},
and will be denoted by $\Pi$.
The concepts of orthogonality and adjointness are
defined on $\Pi$ in the natural way with respect to the
indefinite inner product $[x,y]$.
A subspace $X\subset \Pi$ is called
{\bf positive} if $[f,f]>0$ for
for every non-zero vector $f \in X$,
and {\bf negative} if $[f,f] < 0$ for every non-zero
$f\in X$.
Maximal positive subspaces have dimension
$\kappa(U)$, while maximal negative subspaces have codimension
$\kappa(U)$.

Let $\phi(t)=e^{tA}\phi_0$ be a solution of the evolution equation
$$
\frac{d}{dt} \phi(t) = A\phi(t) \,,\quad \phi(0)=\phi_0\,.
$$
Lyapunov's equation guarantees that
the value of the quadratic form $Q(\phi)=[\phi,\phi]$
strictly increases with $t$,
$$
\frac{d}{dt} Q(\phi(t))
= \bigl\langle (A^*U + U A^*)\phi(t),\phi(t)\bigr\rangle
= \bigl\langle V \phi(t), \phi(t)\bigr\rangle > 0\,.
$$
Denote by $M_+(A)$ the invariant subspace
associated with the part of the
spectrum of $A$ located in the right half plane.
If $\phi\in M_+(A)$, then
$$
Q(\phi) > \lim_{t\to-\infty} Q\bigl(e^{tA}\phi\bigr)=0\,,
$$
which shows that
$M_+(A)$ is a positive subspace of $\Pi$.
A similar argument with $t\to \infty$ shows that
the complementary
subspace $M_-(A)$, which corresponds to the spectrum
of $A$ in the left half plane, is a negative
subspace of $\Pi$.  Since Eq.~(\ref{eq:Lyapunov})
excludes purely imaginary eigenvalues,
these subspaces are maximal, and
consequently $\kappa(A)=\kappa(U)$.

One can also interpret Lyapunov's equation as a dissipativity condition
on $A$ with respect to the Pontryagin space $\Pi$.
In general, a densely defined linear operator $A$ on $\Pi$
called {\bf dissipative } if $\Re[Af,f] \leq 0$ for all $f \in
\Dom(A)$. It is {\bf maximally dissipative} if it has
no proper dissipative extension in $\Pi$.
Assuming Lyapunov's equation, we compute for $\phi\ne 0$
$$
 \Re[A\phi,\phi] = \Re \langle U A \phi,\phi\rangle =
\frac{1}{2}\bigl\langle (A^*U+UA)\phi,\phi\bigr\rangle
=\frac{1}{2}\langle V\phi, \phi\rangle >0\,,
$$
i.e., $-A$ is dissipative.
In this framework, the analogue of Belonosov's theorem was
proven by Azizov~\cite{Azizov} (but note that Azizov
formulates the result in terms of $\Im$ rather than $\Re$):

\begin{theorem}[Azizov]
\label{azizov-theorem} Let $A$ be an operator
on $\Pi$ such that $-A$ is maximally dissipative.
Then there exist a maximal nonnegative subspace $\Pi_+$
and a maximal nonpositive subspace $\Pi_-$ of $\Pi$
such that
$$\Re \sigma(A\vert_{\Pi_+}) \geq 0\,,\quad \Re \sigma(A\vert_{\Pi_-}) \leq 0\,.
$$
Moreover, we can choose $\Pi_+$ and $\Pi_-$ to be
invariant subspaces for $A$, and
$$
\Pi_+\supset M_+(A)\,,\quad \Pi_- \supset M_-(A)\,.
$$

If, additionally, $\Re[A f, f] > 0$ for all
nonzero $f \in \Dom(A)$, then $M_+(A)$ and $M_-(A)$
are themselves maximal positive and negative subspaces for $\Pi$,
respectively, and
$$
M_+(A)\dot{+} M_-(A)=\Pi\,.
$$
\end{theorem}

The second part of Azizov's theorem
implies that $\kappa(A)=\kappa(U)$
provided that $V$ in Eq.~(\ref{eq:Lyapunov})  is positive definite.
This agrees with the conclusion of
Theorem~\ref{belonosov-theorem}, but note the difference
in the hypotheses: Belonosov's assumption that $A$ is sectorial
provides resolvent estimates  that allow
to represent $U$ as a contour integral (thereby proving existence),
and the analytic semigroup $e^{tA}$ appears
in the proof that $\kappa(A)=\kappa(U)$, as sketched above.
In contrast, Azizov's theorem
does not require $A$ to be sectorial, but starts instead from
a given solution to Eq.~(\ref{eq:Lyapunov}).
In the special case where $\kappa(U)=0$,
Theorem~\ref{azizov-theorem} reduces to
a theorem of Phillips that characterizes maximal dissipative
operators as generators of strongly continuous contraction
semigroups. In particular, the spectrum of
$A$ lies in the closed left half plane
(see ~\cite{Yosida}, Corollary 1 in Section IX.4).

In the case where $A$ is a sectorial differential operator
of even order on an interval $[a,b]$
Belonosov proved that the solution of
Lyapunov's equation with $V=I$
is given by a self-adjoint bounded operator~\cite{Belonosov3}.
His results are formulated for ``split'' boundary
conditions that do not couple the values at the two ends of the interval.
Belonosov's results were extended to
second-order sectorial differential operators with
non-split boundary conditions by Tersenov \cite{Tersenov}.
The operators we consider here are of fourth order with
periodic boundary conditions.

It is an interesting open question how to take advantage of the
freedom to choose an arbitrary positive definite self-adjoint bounded
operator $V$ for the right hand side of Eq.~(\ref{eq:Lyapunov}).
For instance, if $A$ is a sectorial non-selfadjoint differential
operator, can $V$ be chosen in such a way that the solution
$U$ is the inverse of a differential operator?

\section{Spaces and norms}
\label{sec:basic}

We start with some estimates
for the differential operator in Eq.~(\ref{eq:def-A}).
We will work in $L^2=L^2\bigl([0,2\pi]\bigr)$, and will use
periodic boundary conditions throughout. The inner product
and norm are denoted by
$$
\langle f,g\rangle=\int_0^{2\pi}f(x)\bar g(x)\,dx\,,\quad ||f||_{L^2}
=\left(\int_0^{2\pi}|f(x)|^2\,dx\right)^{1/2}\,.
$$
For the Fourier coefficients we use the conventions
$$
\hat f(p)=\frac{1}{2\pi}\int_0^{2\pi} f(x)e^{-ipx}\, dx\,,
\quad
f(x)=\sum_{p=-\infty}^\infty \hat f(p) e^{ipx}\,.
$$
In the hope of minimizing confusion,
we will denote functions on $[0,2\pi]$ by lowercase
letters such as  ($f$, $\phi$, \dots),
and functions on the square $[0,2\pi]\times [0,2\pi]$
by uppercase letters ($F$, $\Phi$, ....).
Abusing notation, we will identify a function $F(x,y)\in {\cal L}^2$
with the corresponding integral operator $F$ on $L^2$.
By Schwarz' inequality,
$$
|\langle F\phi, \psi\rangle| = \Bigl\vert \int_0^{2\pi}\int_0^{2\pi}
F(x,y)\phi(y)\bar\psi(x)\, dydx\Bigr\vert \le
||F(x,y)||_{{\cal L}^2}||\phi||_{L^2}||\psi||_{L^2}\,,
$$
and consequently
\begin{equation}\label{eq:Schwarz}
||F||_{L^2\to L^2} \le ||F(x,y)||_{{\cal L}^2}\,.
\end{equation}
Operators on functions of two variables will be denoted by
calligraphic letters $({\cal F}, {\cal G}, \dots)$.
Given a single-variable operator $F$, we denote
by ${\cal F}_x$ or ${\cal F}_y$
the operators that acts on the $x$- or $y$-variable
of a function $\Phi(x,y)$ while keeping the
other one fixed.

On the Sobolev spaces $H^s=H^s\bigl([0,2\pi]\bigr)$
with periodic boundary conditions, we use the norms
$$
||f||_{H^s}^2
= 2\pi \sum_{p=-\infty}^\infty (1+p^4)^{s/2} \, |\hat
f(p)|^2\,.
$$
The corresponding Sobolev spaces of doubly periodic functions
on $[0,2\pi]\times[0,2\pi]$ will be denoted by ${\cal H}^s$,
and their norms are defined by
$$
||F||_{{\cal H}^s}^2 =
4\pi^2\sum_{p,q=-\infty}^\infty (2+p^4+q^4)^{s/2} \,
|\hat F(p,q)|^2\,.
$$
Note that for $s=0$, this agrees with
the definition of the ${\cal L}^2$-norm as the square integral.
The choice of the Fourier multipliers $(1+p^4)^{s/2}$ and
$(2+p^4+q^4)^{s/2}$
in place of the standard $(1+p^2)^s$ and
$(1+p^2+q^2)^s$ allows for an easier
comparison between functions of one and two variables.
Finally, we denote by $D$ the unique positive definite
self-adjoint operator on $L^2$ such that
\begin{equation}\label{eq:def-D}
D^4 \phi = \phi'''' + \phi\,.
\end{equation}
This is a first-order pseudodifferential operator
that provides an isometry from $H^{s+1}$ onto $H^s$ for every value of $s$.

The domain of the operator $A$
in Eq.~(\ref{eq:operator}) consists of periodic
functions in $H^4[0,2\pi]$, and its adjoint is given by
$$
A^*[f]= -f^{''''} - a(x) f'' -b(x)f'-c(x) f\,.
$$
In particular, $A$ is self-adjoint, if $b(x)=a'(x)$.

\begin{lemma} \label{lem:Leibnitz}
For any $a\in H^1$ and every $\phi\in L^2$, we have
$$
||a\phi||_{L^2} \le 0.52\, ||a||_{H^1} ||\phi||_{L^2}\,.
$$
In particular,
$$||A^*+D^4||_{H^2\to L^2}\le M\,,
$$
where $M$ is the constant from Eq.~(\ref{eq:def-M}).
\end{lemma}

\begin{proof} Since
$$
\sum_{p=-\infty}^\infty (1+p^4)^{-1/2} \approx 1.68\,,
$$
we have, for $a\in H^1$,
$$
\sup_{||\phi||_{L^2}=1} ||a\phi||_{L^2}
\le ||a||_{L^\infty}
\le \Bigl(\frac{1}{2\pi} \sum_{p=-\infty}^\infty
   (1+p^4)^{-1/2}\Bigr)^{1/2} ||a||_{H^1}
\le 0.52\,||a||_{H^1}\,.
$$
For the second claim, we use that for $\phi\in H^2$
\begin{eqnarray*}
||(A^*+D^4)\phi||_{L^2}
&\le& ||a(x)\phi''||_{L^2} + ||b(x)\phi'||_{L^2} +||(c(x)-1)\phi||_{L^2}\\
&\le& ||a||_{L^\infty}||\phi||_{H^2} + ||b||_{L^\infty}||\phi||_{H^1}
 + ||c-1||_{L^\infty} ||\phi||_{L^2}\\
&\le& M ||\phi||_{H^2}\,.
\end{eqnarray*}
\end{proof}

\begin{lemma} \label{lem:sector} $A$ is sectorial.
\end{lemma}

\begin{proof} It suffices to show that the {\bf Hausdorff set}
$\bigl \{\langle f, Af \rangle \ \mid\ f\in \Dom(A), ||f||=1  \bigr\}$
is contained in a closed sector
$$
S = \{\lambda_0\} \cup \{ z \in C: |\arg (\lambda_0-z)| \leq \theta\} $$
with some vertex $\lambda_0$ and opening angle $\theta<\pi$,
and that $A-\lambda_0 I$ is invertible (see p. 280 of
\cite{Kato}).

Choose
\begin{eqnarray*}
\lambda_0 &=& \frac{1}{2} \bigl(1+\max_{x}\{-a''(x)+b'(x)-c(x)\}
+ (\max_x [a(x)]_+)^2 \bigr) \\
\theta &=& \tan^{-1} \bigl(\max_x |a'(x)-b(x)|\bigr)\,.
\end{eqnarray*}
We estimate, for $f\in H^2$
\begin{eqnarray}
\notag
\Re\langle f, (\lambda_0-A)f\rangle
&=& \int_0^{2\pi} |f''|^2
-a(x)|f'|^2
+\Bigl(\lambda_0+\frac{1}{2} a''(x) - \frac{1}{2} b'(x) + c(x)\Bigr)|f|^2\,
dx\\
\label{eq:Re}
&\ge & 2\pi \sum_{p=-\infty}^\infty \frac{1+p^4}{2} \,|\hat
f(p)|^2\\
\notag
&=& \frac{1}{2} ||f||_{H^2}^2\,.
\end{eqnarray}
This shows that the spectrum of $A$ lies in the half plane
$\Re\,z<\lambda_0-\frac{1}{2}$.  Similarly,
\begin{eqnarray*}
|\Im\langle f, (\lambda_0-A)f\rangle|
&\le& \int_0^{2\pi} |a'(x)-b(x)|
\, |\Im f'\bar f| \,dx\\
&\le & 2\pi \max_x|a'(x)-b(x)|\sum_{p=-\infty}^\infty |p|\, |\hat f(p)|^2\,.
\end{eqnarray*}
For $||f||=1$ it follows that
$$
\frac{|\Im\, \left\langle f, (\lambda_0-A) f\right\rangle |}
{\Re \left\langle f, (\lambda_0-A)f\right\rangle }
\le \bigl(\max_x|a'(x)-b(x)|\bigr)\,
\Bigl(\sup_{p\in\ZZ}\frac{2|p|}{1+p^4}\Bigr)  = \max_x|a'(x)-b(x)|\,,
$$
which yields the claim.
\end{proof}

\bigskip
The lemma implies that the Cauchy
problem for $A$ has a unique
solution for every initial value
$h_0\in L^2$.  This solution is
analytic in $t$ for $t>0$,
and for any fixed $t > 0$, the function
$h(t, \cdot) \in \Dom(A)$. If the coefficients of $A$ are
analytic, then $h$ is analytic in both variables for $t>0$.
An application of the Lax-Milgram  theorem
similar to Lemma~\ref{lem:Lax-Milgram} below
shows that
$(\lambda_0-A)^{-1}$ maps $L^2$ into $H^2$.
It follows that the resolvent is a compact
operator of the Hilbert-Schmidt type,
and that the spectrum of $A$ is discrete.

\section{The integral kernel $U(x,y)$}
\label{sec:PDE}

Let $A$ be the differential operator from Eq.~(\ref{eq:def-A}).
Theorem~\ref{belonosov-theorem} implies that
Lyapunov's equation has a unique solution $U$, provided that
the spectra of $A$ and $A^*$ are disjoint. Our goal is to
show that $U$ admits an integral representation
$$
U(f) = \int_{0}^{2\pi} U(x,y) f(y)\, dy\,,
$$
and to derive bounds on $U(x,y)$.
Equation~(\ref{eq:Lyapunov}) requires that
$\left\langle (A^* U + UA^*)\phi,  \psi\right\rangle = \langle \phi,
\psi \rangle $
for all smooth periodic test functions functions $\phi,\psi$.
This means that $U(x,y)$ is  a distributional solution of the
partial differential equation~(\ref{eq:PDE-U}).

Let us solve Eq.~(\ref{eq:PDE-U}) in the
special case $A_0=-D^4$, given by
$$
A_0[f]=-f^{''''} - f\,.
$$
By our choice of norms, $-A_0$ defines an isometry from $H^4$
onto $L^2$.
Since $A_0$ has constant
coefficients, the unique solution
can be written as $U_0(x,y)=u_0(x-y)$, where
$$
2 A_0 u_0= \delta_0\,,
$$
in other words, $2U_0(x,y)$ is the Green's function
of $A_0$ on $[0,2\pi]$ with periodic boundary conditions.
One can compute $u_0(x)$  explicitly as a linear combination
$$
u_0(x)= C_1 \cos \frac{x-\pi}{\sqrt2} \cosh\frac{x-\pi}{\sqrt2}
+ C_2 \sin \frac{x-\pi}{\sqrt2} \sinh\frac{x-\pi}{\sqrt2}\,,
$$
where the coefficients are adjusted so that
$u_0$ is periodic and twice differentiable, and its
third derivative jumps by $-1/2$ at $x=0$. From this representation, it is
clear that $U_0$ is smooth away from the line $x=y$,
and that $U_0\in {\cal C}^{2,1}\subset {\cal H}^3$.
Alternately, we easily obtain from the
Fourier representation of $A_0$ that
$\hat u_0(p)=-\frac{1}{4\pi(1+p^4)}$, and
\begin{equation}\label{eq:def-U0}
U_0(x,y)=-\frac{1}{4\pi}\sum_{p=-\infty}^\infty \frac{1}{1+p^4} e^{ip(x-y)}\,.
\end{equation}
In particular, $U_0(x,y)\in H^s$ for all $s<\frac{7}{2}$, and
$||U_0||_{{\cal H}^3} \le 1$.

It remains to analyze the difference
$$
K(x,y):= U(x,y)-U_0(x,y) \,.
$$
By definition, $K$ solves the partial differential equation
\begin{equation}
\label{eq:PDE-K}
{\cal A}^*K(x,y)
=  -\bigl({\cal A}^*-{\cal A}_0\bigr) U_0(x,y)\,.
\end{equation}
The second order differential operator ${\cal A}^*-{\cal A}_0$
maps maps $H^2$ into $L^2$, see Lemma~\ref{lem:Leibnitz}.
A weak solution of this equation is provided by
the next lemma.

\begin{lemma}[Construction of $K$]
\label{lem:Lax-Milgram}
The resolvent of ${\cal A}^*$ is compact and maps ${\cal L}^2$ into
${\cal H}^2$.
\end{lemma}

\begin{proof} Let $\lambda_0$ be the vertex of the sector
computed in Lemma~\ref{lem:sector}, and assume
that $F(x,y)\in{\cal L}^2$.  We verify that the equation
$$ 
\bigl(2\lambda_0 -{\cal A}^*\bigr) K(x,y) = F(x,y)
$$
satisfies the assumptions of the Lax-Milgram theorem, as stated in
[Evans, PDE, p. 297] \cite{Evans}.

Define a bilinear form on
on smooth doubly periodic functions $\Phi, \Psi$
$$
B(\Phi,\Psi)
=\langle (\Phi, (2\lambda_0- {\cal A}^*)\Psi\rangle_{{\cal L}^2}\,.
$$
Then $B$ is extended continuously to ${\cal H}^2$ by
$$
B(\Phi,\Psi)
= \langle \Phi, \Psi\rangle_{{\cal H}^2}
+2 \lambda_0\langle \Phi ,\Psi\rangle_{{\cal L}^2}
-\left\langle \Phi, ({\cal A}^*-{\cal A}_0)
\Psi\right\rangle_{{\cal L}^2}\,.
$$
On the other hand, it follows from Eq.~(\ref{eq:Re}) that
$$
B(\Phi,\Phi)\ge \frac{1}{2}||\Phi||_{{\cal H}^2}^2\,.
$$
Finally, the map
$$
\Phi \mapsto -\langle \Phi,F\rangle_{{\cal L}^2}
$$
defines a continuous linear form on ${\cal H}^2$.
The Lax-Milgram theorem asserts that there exists a unique function
$K(x,y)\in {\cal H}^2$ such that
$$
B(K,\Psi)= \left\langle F, \Psi \right\rangle_{{\cal L}^2}
$$
for all $\Psi\in {\cal H}^2$.
By the resolvent identity, the equation
$$
\bigl({\cal A}^*-\lambda\bigr)K(x,y)=F(x,y)
$$
has a unique weak solution in ${\cal H}^2$
for every value of $\lambda$ that is not an eigenvalue of ${\cal A}^*$
and every $F(x,y)\in {\cal L}^2$.
\end{proof}

\begin{lemma} \label{lem:regularity}
If $K(x,y)\in {\cal H}^2$ solves Eq.~(\ref{eq:PDE-K}),
then $K(x,y)\in {\cal H}^4$, and
$$
||K(x,y)||_{{\cal H}^4} \le 2M ||U_0(x,y)+K(x,y)||_{{\cal H}^2}\,,
$$
where the constant is given by Eq.~(\ref{eq:def-M}).
\end{lemma}

\begin{proof}
If $K(x,y)$ solves Eq.~(\ref{eq:PDE-K}), then
$$
{\cal A}_0 K =-({\cal A}^*-{\cal A}_0)(U_0+K)\,,
$$
and we conclude that
$$
||K(x,y)||_{{\cal H}^4}
\le ||{\cal A}^*-{\cal A}_0||_{{\cal H}^2\to {\cal L}^2}
||U_0+K||_{{\cal H}^2}
\le 2||A^*-A_0||_{H^2\to L^2}||U_0+K||_{{\cal H}^2} \,.
$$
The proof is completed with Lemma~\ref{lem:Leibnitz}.
\end{proof}

\section{Estimates for the operator $U$}
\label{sec:operator}

In this section, we derive bounds for
$U=U_0+K$ as an operator on $L^2$.
Since $K(x,y)\in {\cal H}^4$, while $U_0(x,y)\in H^s$ only
for $s<7/2$, the Fourier coefficients of
$K(x,y)$ decay more quickly than the Fourier
coefficients of $U_0(x,y)$. This in turn implies that
the restriction of $U$ to  high Fourier
modes is dominated by $U_0$. In this section, we
provide the relevant estimates.

As a consequence of the regularity result in
Lemma~\ref{lem:regularity} we see that $U$ defines a bounded linear
operator from $L^2$ to $H^4$, with
$$
||U||_{L^2\to H^4} \le ||U_0||_{L^2\to H^4} + ||K||_{L^2\to H^4}
\le \frac{1}{2} + ||K(x,y)||_{{\cal H}^4}\,.
$$
We have used that $D_x^4U_0=\frac{1}{2}\delta$ and
applied Eq.~(\ref{eq:Schwarz}) to $D^4 K(x,y)$.

One attractive property of the ${\cal H}^4$-norm is that
it depends only on the magnitude of the Fourier
coefficients, not on the phases.
In contrast, the operator norm
$$
||F||_{L^2\to H^4} = \sup_{||\phi||_{L^2}=||\psi||_{L^2}=1}
\langle D^4F\phi, \psi \rangle
= 4\pi^2 \sum_{p,q=-\infty}^\infty
(1+p^4) \hat F(p,q)\hat\phi(q)\hat\psi(p)
$$
can change drastically if we replace $\hat F(p,q)$ by $|\hat
F(p,q)|$. This dependence on cancelations can cause difficulties in
estimates: Multiplying the Fourier coefficients of $F$ with factors
$\alpha(p,q)\in [0,1]$ will not necessarily decrease the operator
norm. On the other hand, the $H^4$-norm provides only a rather loose
bound on the norm of the corresponding integral operator.  For
instance, the kernles $U_0(x,y)$ (and consequently $U(x,y)$ does not
lie in ${\cal H}^4$, even though $||U_0||_{L^2\to H^4}=\frac{1}{2}$.

We find it useful to introduce another norm
on integral kernels that
lies between the ${\cal H}^4$-norm
(as a function of two variables), and the operator
norm (as a linear transformation from $L^2$ to $H^4$).
By construction, this norm depends only
on the modulus of the Fourier coefficients.

\begin{lemma}[Auxiliary norm]
Define, for smooth doubly periodic functions $F$
$$
|||F||| := 4\pi^2 \sup_{||\phi||=||\psi||=1}
\sum_{p,q=-\infty}^\infty (2+p^4+q^4)
|\hat F(p,q)|\,|\hat\phi(p)|\,|\hat \psi(q)|\,.
$$
Then
$$
|||F||| \le ||F(x,y)||_{{\cal H}^4}\,,
$$
and
$$
|||F|||\ge \max\bigl\{ ||F||_{L^2\to H^4},
||F||_{H^{-4}\to L^2}, 2||F||_{H^{-2}\to H^2}\bigr\}\,.
$$
\end{lemma}

\begin{proof}
From the Fourier representation, we see that
\begin{eqnarray*}
|||F|||
&\le& \sup_{||\Phi(x,y)||_{{\cal L}^2}=1}
4\pi^2 \sum_{p,q=i\infty}^\infty
\sum_{p,q=-\infty}^\infty (2+p^4+q^4) |\hat F(p,q)|
\,|\hat\Phi(p,q)|\\
&\le& \sup_{||\Phi(x,y)||_{{\cal L}^2}=1}
\langle {\cal A}_0 F, \Phi\rangle_{{\cal L}^2}\\
&=& ||F(x,y)||_{{\cal H}^4}\,.
\end{eqnarray*}
On the other hand,
$$
||F||_{L^2\to H^4} = ||D^4 F||_{L^2\to L^2}
= \sup_{||\phi||=||\psi||=1}
\sum_{p,q=-\infty}^\infty (1+p^4)\hat\phi(p)\hat \psi(q)\hat F(p,q)
\le |||F|||\,,
$$
and similarly
$$
||F||_{H^{-2}\to H^2}\le \frac{1}{2} |||F|||\,,
\quad ||F||_{H^{-4}\to L^2}\le |||F|||\,.
$$
\end{proof}

We note that if $F$ has positive Fourier coefficients,
then $|||F|||$ agrees with the operator norm of ${\cal A}_0F$
as a linear transformation from $L^2$ into itself.
In particular, $ |||U_0|||=1$.

\begin{lemma} [Tail estimate]
\label{lem:tail}
Assume that $K(x,y)$ solves
Eq.~(\ref{eq:PDE-K}), and let $M$ be given
by Eq.~(\ref{eq:def-M}). Then
$$
|||K-P_NKP_N ||| \le M N^{-2} \bigl(|||U_0+K|||\bigr)\,.
$$
\end{lemma}

\begin{proof} Using Eq.~(\ref{eq:PDE-K})
together with the definition of the norm, we obtain
\begin{eqnarray*}
|||(K-P_NKP_N|||
&=& 4\pi^2 \sup_{||\phi||=||\psi||=1} \
\sum_{ |p|\ge N \ \mbox{\scriptsize or}\ |q|\ge N} \
|\hat \phi(p)|\,|\hat \psi(q)|\,
|\widehat{({\cal A}^*\!-\!{\cal A}_0)U}(p,q)|\\[0.1cm]
&\le& \sup_{|p|\ge N \ \mbox{\scriptsize or}\ |q|\ge N}\
\frac{(1+p^4)^{1/2} + (1+q^4)^{1/2}}{2+p^4+q^4}\,
M |||U|||\\[0.2cm]
&\le& M N^{-2}\bigl(|||U_0+K|||\bigr)\,.
\end{eqnarray*}
\end{proof}

\begin{lemma} \label{lem:Uge1} If $U$ solves Eq.~(\ref{eq:PDE-U}), then
$|||U|||\ge 1$.
\end{lemma}

\begin{proof} Write $U=U_0+K$, and estimate
$$
|||U||| 
\ge |||(I-P_N)U_0(I-P_N)||| - |||(I-P_N)K(I-P_N)|||\,.
$$
The first summand is bounded below by $1$
because ${\cal A}_0U_0=I$, and the second summand is
bounded by $M N^{-2}|||U|||$ according to Lemma~\ref{lem:tail}.
We conclude that
$(1-MN^{-2})|||U|||\ge 1$ for each $N$, and
the claim follows upon taking $N\to\infty$.
\end{proof}

\section{Addition rule for the instability index}
\label{sec:addition}

We return to the Pontryagin space $\Pi$ introduced in
Section~\ref{sec:Lyap}, with the indefinite inner
product given by Eq.~(\ref{eq:innerproduct}).
Let $\Pi_1$ be a finite-dimensional
subspace of $\Pi$, and let
$$
\Pi_2= \Pi_1^{\perp_U}
=\bigl\{f\in\Pi\ \Big\vert [f,g]=0 \ \mbox{for all}\ g\in \Pi_1\bigr\}
$$ be its $U$-orthogonal complement.
By construction, $\dim \Pi_1=\codim \Pi_2$.
The natural question is can we
compute $\kappa(U) $ from the
restrictions $\kappa(U\vert_{\Pi_1})$ and $\kappa(U\vert_{\Pi_2})$?
The difficulty is that $\Pi$ need not be a direct sum
of $\Pi_1$ and $\Pi_2$, because the two subspaces may intersect
non-trivially in a subspace where the quadratic form vanishes.

A subspace $X\subset \Pi$ is called {\bf neutral},
if $[\phi,\phi] = 0$ for all $\phi\in X$.  Two finite-dimensional
neutral subspaces $X$ and $Y$ of $H$
are $\Pi$-{\bf skewly linked}, if $$\dim X=\dim Y$$
and the inner product $[.,.]$ does not degenerate on the direct sum
$X \dot{+} Y$. In particular, no vector of $X$ different
from $0$ is orthogonal to the skewly linked subspace $Y$,
and vice versa.

\begin{theorem}[Theorem 3.4 \cite{iohvidov}]
\label{langer} Let $\Pi_1$ be an arbitrary subspace of $H$,
let $\Pi_2$ be its $U$-orthogonal complement,
and let $X = \Pi_1 \cap \Pi_2$ be their intersection.
There exists a neutral subspace $Y\subset \Pi$ that is skewly linked to $X$
and provides a $U$-orthogonal decomposition
\begin{equation}
\label{decomp} \Pi = \Pi_1' \oplus (X \dot{+} Y) \oplus \Pi_2'\,
\end{equation}
where
$$
\Pi_1 = \Pi_1' \oplus X, \qquad \Pi_2 = \Pi_2' \oplus X.
$$
\end{theorem}

The theorem was originally formulated for the case of {\bf regular}
Pontryagin spaces, where the quadratic form $U$ is a bounded
operator with bounded inverse. Under the assumption that $\Pi_1$ is
finite-dimensional, the result easily extends to the situation where
the inverse of $U$ is unbounded but densely defined. Although the
above decomposition is not unique in general, it yields the
following addition formula for instability indices:

\begin{proposition}
\label{index} Let $\Pi_1$ be a finite-dimensional
subspace of $\Pi$, and let $\Pi_2$ be its $U$-orthogonal complement,
Then its instability index is given by
$$\kappa (U) = \kappa(U\vert_{\Pi_1}) + \kappa(U\vert_{\Pi_2}) +
\dim(\Pi_1\cap \Pi_2)\,.$$
\end{proposition}

\begin{proof}
Theorem~\ref{langer} provides subspaces
$\Pi_1'$ and $\Pi_2'$ such that
$$\kappa (U) = \kappa(U\vert_{\Pi_1'}) + \kappa(U\vert_{\Pi_2'}) +
\kappa(U\vert_{X \dot{+} Y)})\,.
$$
By construction, we have
$\kappa(U\vert_{\Pi_1'})=\kappa(U\vert_{\Pi_1})$
and $\kappa(U\vert_{\Pi_2'})=\kappa(U\vert_{\Pi_2})$.
It remains to compute $\kappa(U\vert_{X \dot{+} Y)})$.

Since $X$ and $Y$ are skewly linked and finite-dimensional,
there exists for each basis $\phi_1, \phi_2, ... \phi_m$ of $X$ a
basis $\psi_1, \psi_2, ... \psi_m$ of $Y$ such that $[\phi_i,\psi_j] =
\delta_{ij}$ ($i,j= 1, ... , m$). By expanding
an arbitrary element
$h\in X\dot{+}Y$ as
$$ h = \sum\limits_{i=1}^{m} \alpha_i \phi_i +
\sum\limits_{j=1}^{m} \beta_j \psi_j\,, $$
the indefinite inner product can be expressed as
$$ [h,h] =  2 \sum\limits_{i=1}^{m} \alpha_i\beta_i
= \frac{1}{2}\left(\sum\limits_{i=1}^{m} (\alpha_i + \beta_i)^2
     - \sum\limits_{i=1}^{m} (\alpha_i - \beta_i)^2 \right)\,. $$
This is an explicit representation of the indefinite inner product
in terms of positive and negative squares, which
shows that $\kappa(U\vert_{X\dot{+}Y}) = \dim(X)$.
\end{proof}

\section{Restriction to finite dimensions}
\label{sec:cutoff}

We first prove the claim in Eq.~(\ref{eq:main-1}).

\begin{proposition} [Projecting out high Fourier modes]
\label{prop:U22}
Let $A$ be given by Eq.~(\ref{eq:def-A}). Assume that
the spectra of $A$ and $-A^*$ are disjoint, and
let $U(x,y)$ be the kernel of the
unique solution of Lyapunov's equation was
constructed in Section~\ref{sec:PDE}.  If
\begin{equation}\label{eq:N-condition1}
N^2> M|||U|||\,,
\end{equation}
where $M$ is the constant from Eq.~(\ref{eq:def-M}), then
$$
\kappa(A) =\kappa(P_NU^{-1}P_N)\,.
$$
\end{proposition}

\begin{proof}  By Theorem~\ref{belonosov-theorem},
we have $\kappa(A)=\kappa(U)$.
Let $[\phi,\psi]=\langle U\phi,\psi\rangle$
be the indefinite inner product associated with $U$.
Choose $\Pi_2$ to be the range of $I-P_N$,
and let $\Pi_1=\Pi_2^{\perp_U}$ be its $U$-orthogonal complement.
We will show that
\begin{equation}
\label{eq:restrict-pi1}
\kappa(U)=\kappa(U\vert_{\Pi_1})\,.
\end{equation}
This will establish the conclusion, because
$$
\langle P_N U^{-1}P_N\phi,\phi\rangle = [U^{-1}P_N\phi, U^{-1}P_N\phi]\,,
$$
and $U$ maps $\Pi_1$ isomorphically onto the range of $P_N$.

Let us apply $[\cdot,\cdot]$ to $D^2 \phi$, where $\phi\in H^2$
and $D$ is given by Eq.~(\ref{eq:def-D}). Writing
$U=U_0+K$, and using that $D^2U_0D^2=-\frac{1}{2}I$, we
see that
$$
\bigl[ D^2\phi, D^2\phi]
= \bigl\langle (U_0+K)D^2\phi, D^2\phi\rangle
= -\frac{1}{2}||\phi||^2 + \langle D^2KD^2 \phi,\phi \rangle\,.
$$
We replace $\phi$ with $(I-P_N)\phi$, and use Lemma~\ref{lem:tail}
to obtain
\begin{eqnarray*}
\bigl[ D^2(I-P_N)\phi, D^2(I-P_N)\phi]
&\le& -\frac{1}{2}\bigl( 1- |||(I-P_N)K(I-P_N)|||\bigr)||\phi||^2\\
&\le& -\frac{1}{2}( 1- \eps_N)||(I-P_N)\phi||^2\,.
\end{eqnarray*}
where $\eps_N=MN^{-2}|||U|||<1$.
It follows that
\begin{equation}
\label{eq:U2-negative}
U\big\vert_{\Pi_2} \le (1-\eps_N) \, U_0\big\vert_{\Pi_2}< 0
\end{equation}
as quadratic forms on $\Pi_2$. In particular,
$\kappa(U\vert_{\Pi_2})=0$, $\Pi_1\cap \Pi_2=\emptyset$,
and Eq.~(\ref{eq:restrict-pi1})
follows with Theorem~\ref{langer}.
\end{proof}

For our final result, we want to
replace $\Pi_1$ by the range of the projection $P_N$
from Eq.~(\ref{eq:def-P}). The next two lemmas
concern the restriction  of $U$ to the range of $P_N$.

\begin{lemma} [Lyapunov equation for $P_NAP_N$]
\label{lem:truncate} Let $A$ be given by Eq.~(\ref{eq:def-A}).
Assume that the spectrum of $A$ and $-A^*$ are disjoint, and
let $U(x,y)$ be the kernel of the
unique solution of Lyapunov's equation that we
constructed in Section~\ref{sec:PDE}.  If
$$
N^4 > M^2|||U|||\,,
$$
then
\begin{equation}\label{eq:Lyap-N}
(P_NAP_N)^*(P_NUP_N) + (P_NUP_N)(P_NAP_N)\ge c_NP_N\,,
\end{equation}
where $c_N=1-M^2N^{-4}|||U|||$. In particular,
$$
\kappa(P_N AP_N)=\kappa(P_N UP_N)\,.
$$
\end{lemma}

\begin{proof} For $\phi\in L^2$, we
write  $\phi_1=P_N\phi$, $\phi_2=(I-P_N)\phi$
and decompose
$$
A =  \left(\begin{array}{cc} A_{11} & A_{12}\\
A_{21} & A_{22}\end{array}\right)\,,\quad
U =  \left(\begin{array}{cc} U_{11} & U_{12}\\
U_{21} & U_{22}\end{array}\right)\,.
$$
From Eq.~(\ref{eq:Lyapunov}), we see that $U_{11}$ solves
Lyapunov's equation
$$
A_{11}^*U_{11} + U_{11} A_{11} = V
$$
with $ V=I_{11} - A^*_{12}U_{21} - U_{12}A_{21}$.
We claim that the right hand
side is positive definite on the range of $P_N$.

To prove this claim,
first observe that we can replace $A$ by $A-A_0$ and $U$ by $K=U-U_0$
in the definition of $V$, because $A_0$ and $U_0$ are
diagonal in the Fourier representation.
We estimate
\begin{eqnarray*}
||A^*_{12}U_{21}+ U_{12}A_{21}||_{L^2\to L^2}
&=& ||P_N(A^*-A_0)(I-P_N)KP_N+ P_NK(I-P_N)(A-A_0)P_N||_{L^2\to L^2}\\[0.1cm]
&\le& \sup_{|p|\ge N \ \mbox{\scriptsize or}\ |q|\ge N}
\frac{(1+p^4)^{1/2}+ (1+q^4)^{1/2}}{2+p^4+q^4}\, M |||K-P_NKP_N|||\\[0.2cm]
&\le& M^2N^{-4}|||U|||
\end{eqnarray*}
by Lemma~\ref{lem:tail}.
It follows that $V \ge c_N P_N >0$
as quadratic forms on the range of $P_N$.
Since $P_NUP_N$ is a finite matrix,
the conclusion of the lemma
follows with Taussky's theorem.
\end{proof}

\begin{lemma} \label{lem:inverse}
Under the assumptions of the previous lemma, $P_NUP_N$ is
invertible on the range of $P_N$, and
$$
||\bigl(D^2P_NUP_ND^2\bigr)^{-1}||_{L^2\to L^2}
\le 2 \frac{1+M}{c_N}\,.
$$
\end{lemma}

\begin{proof}  Let us write
$A_N=D^{-2}P_NAP_ND^2$ and $U_N=D^2P_NUP_ND^2$.
By Eq.~(\ref{eq:Lyap-N}), we have
$$
A_N^*U_N + U_N A_N \ge c_N D^4 P_N
$$
as quadratic forms on the range of $P_N$.
Here, $c_N=1-M^2N^{-4}|||U|||$, as in Lemma~\ref{lem:truncate}.
We apply this inequality to an eigenfunction
$\phi_0$ of $U_N$
$$
2\lambda_0  \Re \langle A_N^*\phi_0,\phi_0\rangle \ge  c_N
\langle D^2\phi_0,D^2\phi_)\rangle\,,
$$
where $\lambda_0$ is the corresponding eigenvalue.
Writing $\phi_0=D^{-2}\psi_0$, and using once more
Lemma~\ref{lem:tail}, we conclude that
$$
||U_N||_{L^2\to L^2}
\le 2 \sup_{\psi\in L^2}
    \frac{ \Re \langle A_N^*D^{-2} \psi,D^{-2}\psi \rangle}
{c_N||\psi||_{L^2}^2}
\le \frac{2}{c_N}||A^*D^{-4}||_{L^2\to L^2} \le
2\frac{1+M}{c_N}\,.
$$
\end{proof}

We are finally ready for our main result.

\begin{proposition} [Projection onto trigonometric polynomials]
\label{prop:main} Let $A$ be  a differential operator given by
Eq.~(\ref{eq:def-A}). Assume that the spectra of $A$ and $-A^*$ are
disjoint, and let $U(x,y)$ be the unique weak solution of
Eq.~(\ref{eq:PDE-U}) in ${\cal H}^2$.  If
\begin{equation}\label{eq:N-condition2}
N^2 > M \bigl( 1 + \sqrt{1+M}\bigr) |||U|||\,,
\end{equation}
where $M$ is given by Eq.~(\ref{eq:def-M}), then
$$
\kappa(A)=\kappa(P_NAP_N)\,.
$$
\end{proposition}

\begin{proof}
Since $U$ solves Lyapunov's equation,
Theorem~\ref{belonosov-theorem} implies that
$\kappa(A)=\kappa(U)$, and we already
know from Lemma~\ref{lem:truncate}
that $\kappa(P_N AP_N) =\kappa(P_NUP_N)$.
We want to apply Proposition~\ref{index} in the case
where $\Pi_1$ is the range of $P_N$.

Since $|||U|||\ge 1$ by Lemma~\ref{lem:Uge1},
our assumption implies that $\delta_N=MN^{-2}<1$.
On
$$
\Pi_2=\mbox{Range}(P_N)^{\perp_U} =
\bigl\{\phi\in L^2\ \vert\
U_{11}\phi_1 + U_{12}\phi_2 = 0 \bigr\}\,,
$$
we compute for the indefinite quadratic form
\begin{eqnarray*}
[\phi,\phi] &=& \langle U_{11}\phi_1,\phi_1\rangle
             + \langle U_{12}\phi_2,\phi_1\rangle
             + \langle U_{21}\phi_1,\phi_2\rangle
             + \langle U_{22}\phi_2,\phi_2\rangle\\
&=& - \bigl\langle U_{21}U_{11}^{-1}U_{12} \phi_2,\phi_2\rangle
             + \langle U_{22}\phi_2,\phi_2\rangle\,.
\end{eqnarray*}
By Eq.~(\ref{eq:U2-negative}) of
Proposition~\ref{prop:U22}, the last term is negative
on the nullspace of $P_N$, and satisfies the bound
$$
D^2 U_{22}D^2 \le -\frac{1-\delta_N|||U|||}{2} (I-P_N)
$$
as quadratic forms.
To estimate the other summand, Lemma~\ref{lem:tail} yields
$$
||D^2U_{21}D^2||_{L^2\to L^2}
\le \frac{1}{2} |||(I-P)UP|||
\le \frac{\delta_N}{2} |||U|||\,,
$$
and analogously
$$
||D^2U_{12}D^2||_{L^2\to L^2}
\le\frac{\delta_N}{2}|||U|||\,.
$$
The middle factor is controlled with
Lemma~\ref{lem:inverse} by
$$
||\bigl(D^2P_NUP_ND^2\bigr)^{-1}||_{L^2\to L^2} \le
2\frac{1+M}{1-\delta_N^2|||U|||}\,.
$$
We arrive at
$$
D^2\bigl\{-U_{21}U_{11}^{-1}U_{12} + U_{22}\bigr\}D^2 \le
-\frac{1}{2} \bigl(1-\delta_N|||U||| -
\frac{\delta_N^2|||U|||^2 }{1-\delta_N^2|||U|||}(1+M)\,.
$$
as quadratic forms. Since $\delta_N^2|||U|||^2 (1+M)
< \bigl(1-\delta_N|||U|||\bigr)^2$ by assumption,
$U$ is negative definite
on $\Pi_2$.  It follows from Proposition~\ref{index}
that $\kappa(U)=\kappa(P_NUP_N)$, completing the proof.
\end{proof}

\section{Numerical examples}
\label{sec:numerical}

Before we look at examples, a few words about
how to verify the hypothesis on $N$ in
Eq.~(\ref{eq:N-condition1}) or
Eq.~(\ref{eq:N-condition2}).  The conditions involve the
solution of the partial
differential equation~\ref{eq:PDE-U}.
A useful consequence of Lemmas~\ref{lem:tail} and~\ref{lem:Uge1} is that
for $\delta_N:=MN^{_2}<1$,
\begin{equation}\label{eq:U-from-UN}
1\le |||U|||\le \frac{1}{1-\delta_N}\bigl(1+|||P_N K|||\bigr)\,.
\end{equation}
This follows by using the triangle inequality
$$
|||U_0+K |||\le 1 + |||P_NKP_N||| + |||K-P_NKP_N|||\,
$$
and solving for $P_NKP_N$ in Lemma~\ref{lem:tail}.

We propose two ways to estimate the size of $|||P_NUP_N|||$.

\begin{itemize}
\item Solve the partial differential equation~(\ref{eq:PDE-U})
by a Galerkin approximation, and use this solution
to compute, approximately, the value of $|||U|||$.
If Eq.~(\ref{eq:N-condition1}) is satisfied for some
value of $N$ much below the dimension of the
Galerkin approximation, we can
apply Proposition~\ref{prop:U22}, and restrict $U$ to
the subspace in Eq.~(\ref{eq:U-complement}).
A basis for this subspace can be computed by
using the Gram-Schmidt algorithm, with the inner product
replaced by the indefinite inner product
associated with $U$. If even Eq.~(\ref{eq:N-condition2})
can be satisfied, then we can just restrict $U$
to the range of $P_N$.

\item Start with a value of $N$
such that $\delta_N=MN^{-2}<1$. Write the matrix $P_NAP_N$ in the
Fourier representation, find its eigenvalues, and bring it into
triangular form. Solve Lyapunov's equation
$$
P_NA^*P_N \tilde U_N + \tilde U_N P_NAP_N=I\,
$$
for $U_N$.  In the Fourier representation, $U_N$ is a finite
matrix. Compute $\lambda_{max}$,
the largest eigenvalue of the matrix
$$
\bigl( (2+p^4+q^4)|U_N(p,q) -1|\bigr)_{|p|,|q|<N}\,.
$$
Then $(1+\lambda_{max})/(1-\delta_N)$ is our best
estimate for  $|||U|||$. If the condition in
Eq.~\ref{eq:N-condition2} holds with the current value of
$N$, we are satisfied and accept the value of $\kappa(A_N)$
as the instability index for $A$.  Else, we increase $N$
accordingly, and repeat the above steps.

\end{itemize}

\bigskip
Proposition~\ref{prop:main} reduces the computation
of the stability index of $A$ to a
finite-dimensional linear algebra problem.
This is illustrated in
Fig.~1 for the particular example the operator $A$ from
\cite{Benilov3}, see Eq.~(\ref{eq:operator}).
In this example, we have
$$
a(x) = 1 + \frac{\alpha_2}{\alpha_3}\sin x\,,
\quad
b(x)= \frac{1-(\alpha_1+\alpha_2)\cos x}{\alpha_3}\,,\quad
c(x) = 0\,.
$$
In place of the constant in Eq.~(\ref{eq:def-M})
we use the slightly smaller value
$$
\tilde M= \sum_{p=-\infty}^\infty
\bigl(|\hat a(p)| +|\hat b(p| +|\hat c(p)|\bigr)
= 2 + \frac{1+\alpha_1+2\alpha_2}{\alpha_3}\, .
$$

The results of our computations are shown in Figure~1.
We see that if the parameter $\alpha_3$ is small,
then the surface tension is not strong enough to overcome
the gravity and the model is unstable with the number of
the unstable eigenvalues growing as the parameter $\alpha_3$ decreases.

\begin{figure}
\begin{center}
\includegraphics[height=5.5cm]{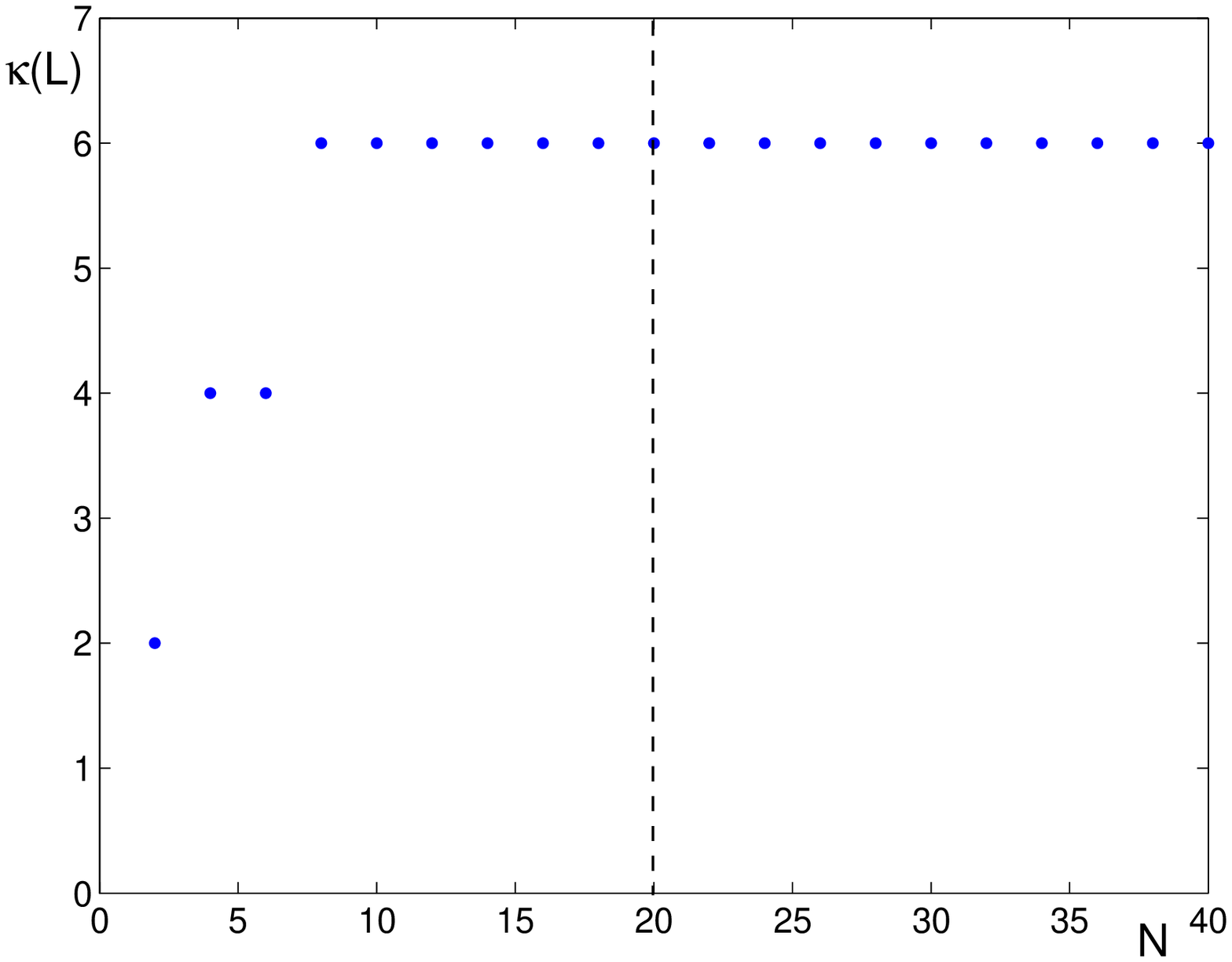}
\includegraphics[height=5.5cm]{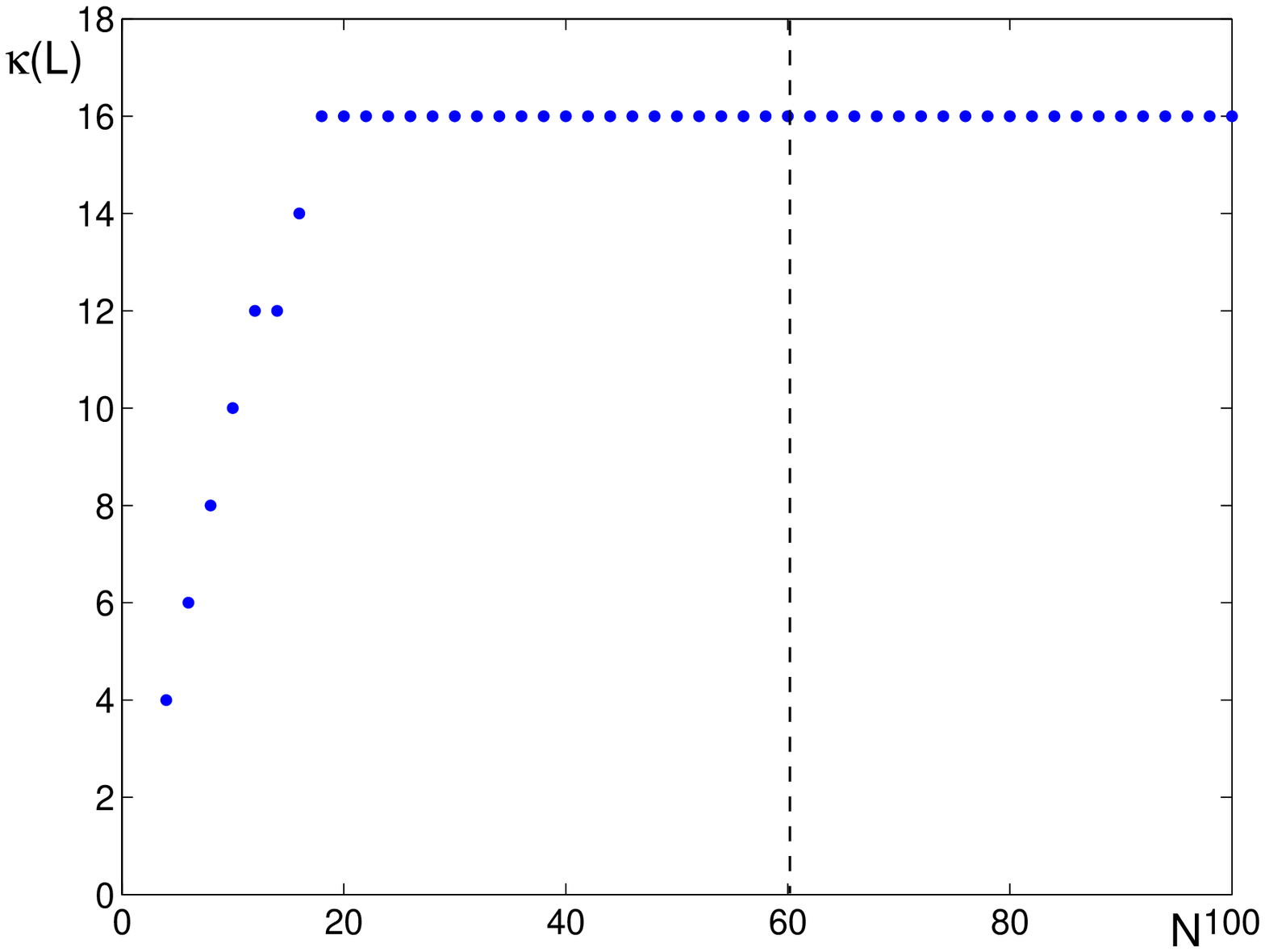}
\end{center}
\caption{On the left the parameters:
  $\alpha_1= 0.0$, $ \alpha_2 = 1$ and
  $\alpha_3 = 0.02$, resulting in $k\approx 190$;
on the right: $\alpha_1 = 0.0$, $\alpha_2 = 1$ and $ \alpha_3 = 0.002$,
resulting in $k \approx 1875$.
The dashed line is showing the suggested cut-off for the dimension
of the finite dimensional subspace which is based on the above
estimations.} \label{fig1}
\end{figure}


\begin{thebibliography}{50}
\bibitem{Azizov} T. Azizov and I.S. Iohvidov,  "Elements of
the theory of linear operators in spaces with indefinite metric"
(Moscow, Nauka, 1986). {\em Translated as:} "Linear operators
in spaces with an indefinite metric", John Wiley \& Sons, 1989.


\bibitem{Belonosov1} V.S. Belonosov, "On instability indices of unbounded operators.
I", Some Applications of Functional Analysis to Problems in
Mathematical Physics [in Russian], Novosibirsk, {\bf 2}, 25--51,
(1984)

\bibitem{Belonosov2} V.S. Belonosov, "On instability indices of unbounded operators.
II", Some Applications of Functional Analysis to Problems in
Mathematical Physics [in Russian], Novosibirsk, {\bf 2}, 5--33,
(1985)

\bibitem{Belonosov3} V.S. Belonosov, "Instability indices of differential operators.
II", Mat. Sb., {\bf 129}, No. 4, 494--514, (1986)

\bibitem{Benilov1} E. S. Benilov, S. B. G. O'Brien, and I. A. Sazonov,
"A new type of instability: explosive disturbances in a liquid film
inside a rotating horizontal cylinder", J. Fluid Mech. {\bf 497},
201--224 (2003)

\bibitem{Benilov2} E. S. Benilov, M. S. Benilov, N. Kopteva, "Steady rimming flows
with surface tension", J. Fluid Mech. {\bf 597}, 91--118 (2008)

\bibitem{Benilov3} E. S. Benilov, N. Kopteva and S. B. G. O'Brien, "Does surface
tension stabilize liquid films inside a rotating horizontal
cylinder", Q. J. Mech. Appl. Math. {\bf 58}, 158--200 (2005)

\bibitem{PT-sym} L. Boulton, M. Levitin, and M. Marletta, "A PT-symmetric periodic problem with
boundary and interior singularities", arXiv:0801.0172v1 [math.SP]
(2008)

\bibitem{ChKP08Preprint} M. Chugunova, I.M. Karabash, S.G. Pyatkov, "On the nature of ill-posedness of the forward-backward heat equation", preprint, arXiv:0803.2552v2 [math.AP] (2008)

\bibitem{ChugStr} M. Chugunova, V. Strauss, "Factorization of the indefinite convection-diffusion
operator", to appear in Math. Reports Acad. Sci. Royal Soc. Canada,
(2008)

\bibitem{Daleckii} Yu. L. Daletskii, M. G. Krein, {\em Stability of
solutions to differential equations in Banach space} [in Russian],
Nauka, Moscow (1970)

\bibitem{D07} E. B. Davies, "An indefinite convection-diffusion operator",  LMS J. Comput. Math.  {\bf 10},
288--306 (2007)


\bibitem{Evans} L. C. Evans, {\em Partial Differential Equations},
 Graduate Studies in. Mathematics, Vol.~19, American Mathematical Society, Rhode Island
 (1998)

\bibitem{Godunov} S. K. Godunov, {\em Modern Aspect of Linear Algebra},
Vol. {\bf 175} (AMS Translations, Providence, 1998)

\bibitem{GohKrein} I. Gohberg and M. G. Krein, {\em Introduction to the theory of linear
non-selfadjoint operators}, Vol. {\bf 18} (AMS Translations,
Providence, 1969)

\bibitem{iohvidov} I.S. Iohvidov, M.G.  Krein, and H. Langer,
{\em Introduction to the spectral theory of operators in spaces with
an indefinite metric} (Mathematische Forschung, Berlin, 1982)

\bibitem{Kato} T. Kato, {\em Perturbation theory for linear
operators}, Springer-Verlag (1976)

\bibitem{Morse} M. Morse, "Variational Analysis: Critical Extremals and Sturmian Extensions", John Wiley and   "
Sons, Inc., New York; London; Sydney; and Toronto (1973)

\bibitem{Pontryagin} L. S. Pontryagin, "Symmetric operators in spaces
with indefinite metric", Izv. USSR, Mat. Ser., {\bf 8}, 243 -- 280,
(1944)

\bibitem{pukhnachev1} V. V. Pukhnachev, "Motion of a liquid film on the surface of a rotating cylinder in a gravitational
field", Zhuranl Prikladnoi Mekhaniki i Tekhnicheskoi Fiziki, {\bf
3}, 78 -- 88, (1977)

\bibitem{pukhnachev2} V. V. Pukhnachev,
"Capillary/gravity film flows on the surface of a rotating
cylinder", Journal of Mathematical Science

\bibitem{Taussky} O. Taussky,
"A generalization of a theorem of Lyapunov",
J. Soc. Indust. Appl. Math. 9 1961 640--643.

\bibitem{Tersenov} Ar. S. Tersenov, "Solvability of the Lyapunov equation for non-selfadjoint second-order differential operators
with nonlocal boundary conditions ", Siberian Mathematical Journal,
{\bf 39}, No. 5, 1026 -- 1042, (1998)

\bibitem{Trefethen} L.N. Trefethen,
{\em Spectral Methods in Matlab} (SIAM, Philadelphia, 2000)

\bibitem{Yosida} K. Yosida, "Functional Analysis", 6th ed., Springer 1980.

\bibitem{Weir1} J. Weir, "An indefinite convection-diffusion operator with real spectrum",   arXiv:0711.1371v1
[math.SP] (2007)

\end{thebibliography}
\end{document}